\documentclass[11pt]{article}
\usepackage{amssymb,amsmath,bm}
\usepackage[colorlinks=true, urlcolor=blue,linkcolor=blue, citecolor=blue]{hyperref}
\usepackage{mathrsfs,verbatim}
\usepackage{caption}
\usepackage{constants,color}
\usepackage{enumerate}
\usepackage{bbm}
\usepackage{tikz}
\usepackage{tkz-tab}
\usepackage{tikz-qtree}
\usepackage{appendix}
\newcommand{\E}{\mathbb{E}}
\newcommand{\Pro}{\mathbb{P}}
\newcommand{\Var}{\mathrm{Var}}
\newcommand{\Cov}{\mathrm{Cov}}

\newcommand{\1}{\mathbf{1}}

\begin{document}
\newcommand{\bea}{\begin{eqnarray}}
\newcommand{\ena}{\end{eqnarray}}
\newcommand{\beas}{\begin{eqnarray*}}
\newcommand{\enas}{\end{eqnarray*}}
\newcommand{\beq}{\begin{equation}}
\newcommand{\enq}{\end{equation}}
\def\qed{\hfill \mbox{\rule{0.5em}{0.5em}}}
\newcommand{\bbox}{\hfill $\Box$}
\newcommand{\ignore}[1]{}
\newcommand{\ignorex}[1]{#1}
\newcommand{\wtilde}[1]{\widetilde{#1}}
\newcommand{\mq}[1]{\mbox{#1}\quad}
\newcommand{\bs}[1]{\boldsymbol{#1}}
\newcommand{\qmq}[1]{\quad\mbox{#1}\quad}
\newcommand{\qm}[1]{\quad\mbox{#1}}
\newcommand{\nn}{\nonumber}
\newcommand{\Bvert}{\left\vert\vphantom{\frac{1}{1}}\right.}
\newcommand{\To}{\rightarrow}
\newcommand{\supp}{\mbox{supp}}
\newcommand{\law}{{\cal L}}
\newcommand{\Z}{\mathbb{Z}}
%Erd\"{o}s-R\'{e}nyi

\newcommand{\N}{\mathbb{N}}
\newcommand{\T}{\mathcal{T}}

\newcommand{\ucolor}[1]{\textcolor{blue}{#1}}  % Umit's colored text
\newcommand{\ucomm}[1]{\marginpar{\tiny\ucolor{#1}}}  % Umit's marginpars

\newcommand{\ecolor}[1]{\textcolor{red}{#1}}  % Ella's colored text
\newcommand{\ecomm}[1]{\marginpar{\tiny\ecolor{#1}}}  % Ella's marginpars

 \newtheorem{theorem}{Theorem}[section]
\newtheorem{corollary}{Corollary}[section]
\newtheorem{conjecture}{Conjecture}[section]
\newtheorem{proposition}{Proposition}[section]
\newtheorem{lemma}{Lemma}[section]
\newtheorem{definition}{Definition}[section]
\newtheorem{example}{Example}[section]
\newtheorem{remark}{Remark}[section]
\newtheorem{case}{Case}[section]
\newtheorem{condition}{Condition}[section]
\newcommand{\pf}{\noindent {\it Proof:} }
\newcommand{\proof}{\noindent {\it Proof:} }
\frenchspacing

%\affiliation[*]{University of Southern California}

\title{\bf  Asymptotic results on Hoppe trees and its variations}
\author{Ella Hiesmayr\footnote{Ko\c{c} University, Istanbul, Turkey. email: ehiesmayr17@ku.edu.tr} \hspace{0.2in} \"{U}m\.{i}t I\c{s}lak\footnote{Bo\u{g}azi\c{c}i University, Istanbul, Turkey. email: umit.islak1@boun.edu.tr}} \vspace{0.25in}

\maketitle

\begin{abstract}
A uniform recursive tree on $n$ vertices is a random tree where each possible $(n-1)!$ labeled recursive rooted tree is selected with equal probability. In this paper we introduce and study weighted trees, a non-uniform recursive tree model departing from the recently introduced Hoppe trees. This class generalizes both uniform recursive trees and Hoppe trees. The generalization provides diversity among the nodes, making the model more flexible for applications. We also analyze the number of leaves, the height, the depth, the number of branches, and the size of the largest branch in these weighted trees. 
\bigskip


\bigskip

\end{abstract}

\section{Introduction}\label{sec:intro}

A uniform recursive tree (URT) on $n$ vertices is a rooted random recursive tree where each possible $(n-1)!$ distinct trees has the same probability of appearing. Another way of looking at uniform recursive trees is that one starts constructing the tree with only the root (node 1) and node 2 attached to the root. Afterwards, at each step $k = 3,\ldots,n$, node $k$ connects to one of the prior nodes $j$ with equal probability $1 / (k-1)$. A detailed survey on URTs can be found in \cite{Mahmoud}, and a book length treatment of random trees can be found in \cite{Drmota09}. 

URTs, although pretty simple to construct, have been used in various applications.  These include but are not restricted to the spread of epidemics \cite{Feng05},  determining the genealogy of ancient and medieval texts \cite{NajockHeyde82}, analyzing pyramid schemes \cite{Gastwirth}, and the spread of a fire in a tree \cite{Fire}. Though these investigations use uniformity in their models, having different distributions would provide a lot more flexibility to the  researcher.  Using URTs implies that all nodes are identical in a certain sense,   for example, in the spread of epidemics, that every infected person is equally likely to infect the next one, or in the study of medieval texts, that every book is equally likely to be copied. This is obviously not the case in real world applications. 

Parallel to the development of the theory of uniform recursive trees, various other recursive tree structures have already been studied. One of the most well known are binary recursive trees, which are described in \cite{Flajolet}. Other non-uniform recursive tree models include  plane-oriented recursive trees \cite{Szymanski87}, scaled attachment random recursive trees \cite{Devroye11} and biased recursive trees \cite{Altokislak}. 

Our interest here is on another    natural generalization, Hoppe trees,  that was recently considered in \cite{Hoppe}. There, the root is assigned a weight $\theta$, all other nodes get weight 1. Node $i$ then attaches to the root with probability $\frac{\theta}{\theta+i-2}$ and to any other node with probability $\frac{1}{\theta+i-2}$. This model is associated to Hoppe's urn, which has an application in modelling the alleles of a gene with mutation rate $\theta>0$. Concerning many properties like the number of leaves, the height and the depth of node $n$, Hoppe trees behave similarly to uniform recursive trees in an asymptotic sense.

The  model  in this paper generalizes the idea of Hoppe trees: we  assign every node a weight $w_i$. Node $j$ then attaches to node $1 \leq i<j$ with probability $\frac{w_i}{w_1 + \cdots + w_{j-1}}$.  We call the resulting tree construction a weighted recursive tree (WRT).

Introducing weights is also interesting from the point of view of applications since it allows to introduce diversity among the nodes. In  other non-uniform models  discussed above, all nodes have the same behaviour, or in other words attract nodes according to the same rule. When a recursive process does not satisfy such conditions, weighted recursive trees can be used to model it more precisely.  Moreover the properties of weighted recursive trees and how much they differ from the uniform model can be interpreted as an indicator for the stability of a process. It is reasonable to assume that it is in general more probable for some nodes to get children as others. For example some persons might be more likely to infect others, some copies of ancient texts are more probable to have been copied again and some people might be more likely to recruit new people. Thus it is interesting to see how much fluctuation  in the attachment probabilities can be tolerated when using the uniform model.  

Below, for  the  generalized model of this note we first give a coupling construction to construct a WRT  on $n$ nodes from a  Hoppe  tree. This allows us to understand statistics such as the height of the resulting random tree. We  then study the number of branches and the depth of node $n$ and give their expectation and variance, as well as some conditions under which asymptotic normality holds. We moreover derive explicit values for the expectation and the variance for certain examples of weight sequences.

The rest of the paper is organized  as follows. Next section introduces a coupling used to construct a  WRT from a Hoppe tree.  Section \ref{sec:couplingapplications} applies the coupling construction for an analysis of the number of leaves, the height and the size of the largest branch.  In Section \ref{sec:largestbranch}, we relate Hoppe trees to Hoppe permutations, and use the coupling construction of Section \ref{sec:coupling} to study the size of largest branch in WRTs. 
   In Sections \ref{sec:depth} and \ref{sec:branches}, we study the depth of a WRT and  the number of branches in WRT. 
   
\section{A useful coupling construction}\label{sec:coupling}

%1st coupling
\subsection{Constructing WRTs from URTs}
\label{rem:URT2mkRT} We will first introduce a coupling allowing us to construct a special kind of WRT from a URT. We will not use this coupling in analysis of random tree statistics, because the second coupling construction, that is to introduced below,   can be applied to a more general class of WRTs. We still wanted to introduce this version because it is based on not Hoppe trees but URTs, a much better studied structure, as a starting point.

 We will call trees that have a weight sequence such that the first $k$ nodes have a constant weight equal to $\theta$ and all other nodes have weight 1, $\theta^k$-RTs. It is possible to construct a $\theta^k$-RT from a URT by a coupling construction when $\theta \in \mathbb{N}^+$. To emphasize this assumption we will use $m$ instead of $\theta$ in this part. In particular we can construct Hoppe trees for which the weight of the root is a natural number by this coupling. 
To avoid confusion let us denote the nodes in the URT by $i$ and the nodes in the reconstructed tree by $i^*$. 

The coupling construction in that case goes as follows: First construct a URT with $mk + n-k$ nodes. We write $\mathcal{T}_{mk+n-k}$ for this URT. Since we want the weight of the first $k$ nodes to be $m$ we then join several nodes into one in the following way:
\begin{itemize}
\item[\boldmath$\cdot$] Nodes $1, \dots, m$ become node $1^*$,
\item[\boldmath$\cdot$] $m+1, \dots, 2m$ become node $2^*$ \dots 
\item[\boldmath$\cdot$] $(k-1)m +1, \dots, km$ become node $k^*$.
\end{itemize}

The new node $i^*$ gets all the children of $(i-1)m+1, \dots, im$. But since we joined several nodes into one and the nodes $(j-1)m +1, \dots, jm$ might have different parents, for $1<j\leq k$, we set the parent of $j^*$ as the parent of $(j-1)m+1$, i.e. of the node with the smallest label among those that become $j^*$. Thus, if in the URT the parent of $(j-1)m+1$ is any of the nodes $(i-1)m+1, \dots, im$, the parent of $j^*$ is $i^*$.

For $j >k$, we set $j^* = j+k(m-1)$, so all nodes after $k$ only correspond to a single node, we just need to "translate" the names of the nodes to take into account that we used $mk$ nodes instead of $k$ for the first $k$ nodes in the reconstructed tree. If for $j>k$ the parent of node $j+k(m-1)$  is among the first $km$ nodes of the URT, we check into which range this node falls and the parent of $j^*$ is chosen as above. In other words if, for $1 \leq i \leq k$, the parent of node $j+k(m-1)$ is one of $(i-1)m+1, \dots, im$, the parent of node $j^*$ is node $i^*$. If the parent of node $j+k(m-1)$ is equal to $h+k(m-1)$ with $h>k$, the parent of node $j^*$ is node $h^*$.
We call the tree we thus obtain $\mathcal{T}_n^{m^k}$.
It can be easily verified that the obtained attachment probabilities correspond to the ones of an $m^k$-tree, for details see \cite{ellatez}.

We now show, as an example, how the coupling can be used to study the number of leaves, i.e. the number of nodes without children, of a $\theta^k$-RT. Let $\mathcal{L}_{mk+n-k}$ denote the number of leaves of $\mathcal{T}_{mk+n-k}$ and $\mathcal{L}_n^{m^k}$ denote the number of leaves of $\mathcal{T}_{n}^{m^k}.$
Then $\mathcal{L}_{mk+n-k}$ can be used to bound $\mathcal{L}_n^{m^k}$. First of all if a node $i > km $ is a leaf in $\mathcal{T}_{mk+n-k}$, the corresponding node in $\mathcal{T}_n^{m^k}$, which is ${i-k(m-1)}^*$, is also a leaf.  The reconstruction process thus only affects the children of the nodes $i^*$ with $1 \leq i \leq k$ and the root cannot be a leaf, so we can have at most $k-1$ additional leaves. Moreover for each $2 \leq i \leq k$ we can at most "loose" $m-1$ leaves since if all $(i-1)m+1, \dots, im$ are leaves in $\mathcal{T}_{mk+n-k}$, $i^*$ will be a leaf in $\mathcal{T}_n^{m^k}$ too.  Hence we can conclude that 
\begin{equation*}
\mathcal{L}_{mk+n-k}-k(m-1)< \mathcal{L}_n^{m^k} <\mathcal{L}_{mk+n-k}+ k-1  .
\end{equation*}
Together with results about the expected number of leaves of URTs this implies after some simple manipulations that 
\begin{equation*}
\begin{split}
 \left |\E[\mathcal{L}_n^{m^k}]-\E[\mathcal{L}_n] \right | \leq \frac{k(m+1)}{2}, \\
\end{split}
\end{equation*}
where $\mathcal{L}_n$ denotes the number of leaves in a URT on $n$ nodes. It is possible to derive other results from this coupling, but since the second coupling we now present is more general, we will not go further into it here.

%2nd coupling

\subsection{Constructing WRTs from Hoppe trees}

We will now introduce a coupling construction for WRTs whose nodes have constant weight after a certain index, a class similar to, but more general, than $\theta^k$-RTs. Let $\T_n^w$ be a WRT and $(w_i)_{i \in \N}$, the weight sequence of $\T_n^w$, be such that there is a $k \in \N$ such that $w_i=1$ for all $i>k$. Then we can construct $\T_n^w$ from a Hoppe tree with root weight $\theta=\sum_{i=1}^{k} w_i$ by a  coupling construction, more precisely by splitting the root into $k$ nodes. To avoid confusion we will write $i$ for node $i$ in the Hoppe tree and $i^*$ for node $i$ in the reconstructed tree.

We now describe the coupling construction: First construct a Hoppe tree on $n-k+1$ nodes and with $\theta$, the weight of the root, equal to $\sum_{i=1}^{k} w_i$. Then construct a WRT of size $k$ corresponding to $(w_i)_{i \in \mathbb{N}}$. Now we replace the root of the Hoppe tree by this weighted recursive tree of size $k$ in the following way: Node $1^*$, $2^*$, \dots, $k^*$ are the nodes of the WRT of size $k$ we just constructed.
For $i \geq 2$, node $i$ in the Hoppe tree becomes node $i+k-1^*$ in the reconstructed tree, so we shift the names of the rest of the nodes by $k-1$. Then for all $i \geq 2$, if $i$ is a child of 1 in the Hoppe tree, $i+k-1^*$ becomes a child of one of the nodes $1^*, \dots, k^*$ in the reconstructed tree, proportionally to their weights. This means that if $i$ is a child of 1 in the Hoppe tree, for $1\leq j \leq k$, node $i+k-1^*$ will become a child of a node $j^*$  in the reconstructed tree with probability $\frac{w_{j}}{\sum_{\ell=1}^{k} w_{\ell}}$. 

Let us check that this gives the attachment probabilities corresponding to the WRT we aim to construct.
\begin{itemize}
\item[\boldmath{$\cdot$}] For $1\leq i \leq k<j$,
\begin{equation*}
\begin{split}
\Pro & (j^* \text{ is child of } i^*)  = \Pro(j-k+1 \text{ is child of } 1, i^*\text{  is chosen as the parent of } j^*) \\
&= \frac{\sum_{\ell=1}^{k} w_{\ell}}{j-k+1 -2 +\sum_{\ell=1}^{k} w_{\ell}} \frac{w_i}{\sum_{\ell=1}^{k} w_{\ell}} = \frac{w_i}{j-1 -k+\sum_{\ell=1}^{k} w_{\ell} }. \\
\end{split}
\end{equation*}
\item[\boldmath{$\cdot$}] For $k<i<j$, 
\begin{equation*}
\begin{split}
\Pro & (j^* \text{ is child of } i^*)
 = \Pro(j-k+1 \text{ is child of } i-k+1)\\
&= \frac{1}{j-k+1 -2 +\sum_{\ell=1}^{k} w_{\ell}} = \frac{1}{j-1-k+\sum_{\ell=1}^{k} w_{\ell}}. \\
\end{split}
\end{equation*}
\end{itemize}

The following two sections will be using the coupling construction just described.

\section{Use of the coupling in WRT statistics}\label{sec:couplingapplications}

In this section we apply the coupling construction of the previous section to study the number of leaves in a WRT and the height of a Hoppe tree. Later, in Section \ref{sec:largestbranch}, the coupling construction will also be used in order to understand the size of the largest branch in a WRT. 

\subsection{Number of leaves}

A node in a tree with degree one is said to be a leaf. 
Focusing on the number of leaves, the reconstruction process in our coupling construction implies that all the leaves of the Hoppe tree are still leaves in the reconstructed tree, since we do not change any relation among the nodes $2, \dots, n-k+1$ of the Hoppe tree or respectively $k+1^*, \dots n^*$ of the reconstructed tree. There can be at most $k-1$ additional leaves among the first $k$ nodes. Thus we can bound the number of leaves $ \mathcal{L}_n^{w} $ of $\T^w$ by the number of leaves  $\mathcal{L}_{n}^{\theta}$ of $\T^\theta$:
\begin{equation}  \label{eq:CouplingLeaves}
\mathcal{L}_{n-k+1}^{\theta} \leq \mathcal{L}_n^{w} \leq  \mathcal{L}_{n-k+1}^{\theta} +k-1. \\
\end{equation}. 

In \cite{Hoppe} the following results about the leaves of Hoppe trees are given.

\begin{theorem}[\cite{Hoppe}] \label{thm:HoppeLeaves}
Let $\mathcal{L}_n^{\theta}$ denote the number of leaves of a Hoppe tree with $n \geq 2 $ nodes. Then
\begin{equation}
\begin{split}
&\E [\mathcal{L}_n^{\theta}] = \frac{n}{2} + \frac{\theta -1}{2} + \mathcal{O}\left (\frac{1}{n} \right), \\
&\Var (\mathcal{L}_n^{\theta}) = \frac{n}{12} + \frac{\theta -1}{12} + \mathcal{O}\left (\frac{1}{n} \right ), \\
& \Pro(|\mathcal{L}_n^{\theta} - \E[\mathcal{L}_n^{\theta}]| \geq t) \leq 2 e^{-\frac{6t^2}{n+\theta+1}} \text{ for all } t>0 \text{ and } \\
&\frac{\mathcal{L}_n^{\theta} - \E[\mathcal{L}_n^{\theta}]}{\sqrt{\Var\left (\mathcal{L}_n^{\theta}\right)}} \xrightarrow[d]{n \to \infty} \mathcal{G}.
\end{split}
\end{equation}
\end{theorem}

Using the above theorem we can thus derive results on the number of leaves of WRTs whose nodes having constant weight after a certain index.

\begin{theorem} \label{thm:WRTLeaves}
Let $\mathcal{L}^{w}_{n}$ denote the number of leaves of a WRT of size $n$ with weight sequence $(w_i)_{i \in \mathbb{R}}$ such that there is a $k \in \mathbb{N}$ such that for all $i>k$ we have $w_i = 1$. Then
\begin{equation}
\begin{split}
&\E [\mathcal{L}_n^{w}] = \frac{n}{2} + C + \mathcal{O} \left ( \frac{1}{n} \right ) \text{ with } |C| \leq \frac{\sum_{i=1}^{k} w_i +k}{2},\\
&\Var(\mathcal{L}_n^{w})) = \frac{n}{12} + \mathcal{O} (\sqrt{n}), \\
& \Pro (|\mathcal{L}_n^{w} - \E[\mathcal{L}_n^{w}]|\geq t) \leq 2 e^{- \frac{6(t-2k+2)^2}{n-k+2+ \sum_{i=1}^k w_i}} \text{ for all } t>0 \text{ and } \\
& \frac{\mathcal{L}^{w}_{n} - \E[\mathcal{L}^{w}_{n}]}{\Var(\mathcal{L}^{w}_{n})} \longrightarrow_d \mathcal{G} \text{ as } n \to \infty.
\end{split}
\end{equation}
\end{theorem}

\begin{proof}
First of all for the expected value we get from Theorem \ref{thm:HoppeLeaves} and \eqref{eq:CouplingLeaves}.
\begin{equation*}
\begin{split}
&\mathcal{L}_{n-k+1}^{\theta} \leq \mathcal{L}_n^{w} \leq  \mathcal{L}_{n-k+1}^{\theta} +k-1 \\
& \Rightarrow \frac{n-k+1}{2} + \frac{\sum_{i=1}^{k} w_i-1}{2} + \mathcal{O}\left ( \frac{1}{n} \right ) \leq \E\left [\mathcal{L}_n^{w}\right ] \\
& \hspace{7em} \leq \frac{n-k+1}{2} + \frac{\sum_{i=1}^{k} w_i-1}{2} + k -1 + \mathcal{O}\left ( \frac{1}{n} \right ) \\
& \Rightarrow \frac{n}{2} + \frac{\sum_{i=1}^{k} w_i-k}{2} + \mathcal{O}\left ( \frac{1}{n} \right ) \leq \E \left [\mathcal{L}_n^{w}\right ] \leq \frac{n}{2} + \frac{\sum_{i=1}^{k} w_i+k-2}{2} + \mathcal{O}\left ( \frac{1}{n} \right ) \\
& \Rightarrow \E\left [\mathcal{L}_n^{w}\right ] = \frac{n}{2} + C + \mathcal{O}\left ( \frac{1}{n} \right ), 
\end{split}
\end{equation*}
where $C$ depends on $k$ and $w_1, \dots, w_k$, and we have $|C| \leq \frac{\sum_{i=1}^{k} w_i +k}{2}$.

Equation \eqref{eq:CouplingLeaves} also directly gives a concentration inequality. We have that

\begin{equation*}
\begin{split}
\Pro & (|\mathcal{L}_n^{w} - \E[\mathcal{L}_n^{w}]|\geq t) \\
& \leq \Pro( |\mathcal{L}_n^{w} - \mathcal{L}_{n-k+1}^{\theta} | +  |\mathcal{L}_{n-k+1}^{\theta} - \E  [\mathcal{L}_{n-k+1}^{\theta}  ] | +  | \E [\mathcal{L}_{n-k+1}^{\theta} ] - \E  [\mathcal{L}_n^{w}  ]  |\geq t)  \\
& \leq \Pro(|\mathcal{L}_{n-k+1}^{\theta} - \E  [\mathcal{L}_{n-k+1}^{\theta} ]  | \geq t-2k+2)  \\
& \leq 2 e^{- \frac{6(t-2k+2)^2}{n-k+2+ \sum_{i=1}^k w_i}}.
\end{split}
\end{equation*}

The coupling also gives us an approximation for the variance. Let $Y$ denote the number of additional leaves among the first $k$ nodes in the reconstructed tree. Then $\Var(Y) = \mathcal{O}(k)$ since $Y \leq k$. Since
$\Var(\mathcal{L}_n^{w}) = \Var(\mathcal{L}_{n-k+1}^{\theta} +Y ) = \Var(\mathcal{L}_{n-k+1}^{\theta}) + \Var(Y) + \Cov(\mathcal{L}_{n-k+1}^{\theta} ,Y)$
we get by the Cauchy-Schwarz inequality
\begin{equation*}
\Var(\mathcal{L}_n^{w})) = \frac{n}{12} + \mathcal{O} (\sqrt{n}).
\end{equation*}

In a similar way, one can make conclusions about the  asymptotic distribution. For this we will need Slutsky's theorem:
Let $X_n$ and $Y_n$ be sequences of random variables such that $X_n \to_d X$ and $Y_n \to_d c$ for $c \in \mathbb{R}$. Then
$$ \lim_{n\to \infty} X_n + Y_n  =_d  X +c.$$

In order to derive a central limit theorem for $\mathcal{L}^{w}_{n}$, we write 
\begin{equation*}\frac{\mathcal{L}^{w}_{n}- \E[\mathcal{L}_{n}^w]}{\sqrt{\Var(\mathcal{L}_{n}^w)}} 
= \frac{\mathcal{L}^{w}_{n}- \mathcal{L}_{n-k+1}^\theta}{\sqrt{\Var(\mathcal{L}_{n}^w)}} 
+   \frac{\mathcal{L}^{\theta}_{n-k+1}- \E[\mathcal{L}_{n-k+1}^\theta]}{\sqrt{\Var(\mathcal{L}_{n}^w)}} 
+ \frac{\E[\mathcal{L}^{\theta}_{n-k+1}]- \E[\mathcal{L}_{n}^w]}{\sqrt{\Var(\mathcal{L}_{n}^w)}}.
\end{equation*}
Now  by Theorem \ref{thm:HoppeLeaves} and our previous result on $\Var(\mathcal{L}_n^w)$ we have
\begin{equation*}
\frac{\mathcal{L}^{\theta}_{n-k+1}- \E[\mathcal{L}_{n-k+1}^\theta]}{\sqrt{\Var(\mathcal{L}_{n-k+1}^\theta)}}\frac{\Var(\mathcal{L}_{n-k+1}^\theta)}{\Var(\mathcal{L}_{n}^w)} \xrightarrow[d] {n \to \infty} \mathcal{G},\end{equation*}
and since $\Var(\mathcal{L}_{n}^w) \longrightarrow \infty$ as $n \to \infty$ we have by \eqref{eq:CouplingLeaves} that  
\begin{equation*}\left | \frac{\mathcal{L}^{w}_{n}- \mathcal{L}_{n-k+1}^\theta}{\sqrt{\Var(\mathcal{L}_{n}^w)}} + \frac{\E[\mathcal{L}^{\theta}_{n-k+1}]- \E[\mathcal{L}_{n}^w]}{\sqrt{\Var(\mathcal{L}_{n}^w)}} \right | \leq \frac{2k}{\sqrt{\Var(\mathcal{L}_{n}^w)}} \xrightarrow{a.s.} 0.
\end{equation*}
Now we can apply Slutsky's theorem and conclude that
\begin{equation*} \frac{\mathcal{L}^{w}_{n}- \E[\mathcal{L}_{n}^w]}{\sqrt{\Var(\mathcal{L}_{n}^w)}} \xrightarrow[d] {n\to \infty} \mathcal{G}.
\end{equation*}
\hfill $\square$
\end{proof}
 
\begin{remark} \label{rem:leavesProb} It might be possible to get results on the number of leaves of a general WRT $\mathcal{T}_n^w$ by writing $\mathcal{L}_{n}^{w}$ as the sum of $\1(\ell_i^{w})$ where $\ell_i^{w}$ denotes the event that $i$ is a leaf in $\mathcal{T}_n^w$. It follows from the construction principle that
\begin{equation*}
\Pro \left (\ell_i^{w}\right ) = \prod_{j=i+1}^{n} \left ( 1- \frac{w_i}{w_1 + \dots + w_{j-1}} \right ).
\end{equation*}
After some manipulation this expression becomes
\begin{equation*}
\Pro \left (l_i^{w}\right )  = \frac{ w_1 + \dots + w_{i-1}}{w_1 + \dots + w_{n-1}} \prod_{j=i+1}^{n-1} \left ( 1 + \frac{w_{j} - w_{i}}{w_1 + \dots + w_{j-1}} \right ).
\end{equation*}
An exact expression for the expectation of the number of leaves of a $\theta^k$-RT can be obtained by writing $\E\left [\mathcal{L}_n^{\theta^k}\right ]  = \sum_{i=2}^{n} \E\left [\1\left (\ell_i^{\theta^k}\right )\right ]$ and using the  expression above. After some computations we get by this method
\begin{equation*}\E\left [\mathcal{L}_n^{\theta^k} \right ] = \frac{n}{2} + \frac{k(\theta-1)}{2} + \frac{k \theta(1-k\theta)}{2(k (\theta-1) + n-1)}+ \frac{k-1}{2} \prod_{i=1}^{n-1-k} \frac{\theta (k-1)+i}{\theta k +i}.
\end{equation*}
\end{remark}

\begin{remark} 
For $\theta^k$-RTs it is also possible to obtain results about the expectation, variance and concentration rate of the number of leaves by using a martingale argument similar to the one in \cite{Hoppe}, for details see \cite {ellatez}.
\end{remark}

\subsection{Height}

As a second example, we discuss the height of a WRT,  which   is defined as  the length of the longest path from the root to a leaf.  Let $\mathcal{H}_n^w$ denote the height of a WRT with weight sequence $(w_i)_{i \in \mathbb{N}}$ such that $w_i =1$ for $i>k$.  Let moreover  $\mathcal{D}_{1,i}^w$ and $\mathcal{D}_{1,i}^\theta$ denote the distance between the root and node $i$ in the reconstructed weighted tree and the original Hoppe tree respectively. For $i \leq k$, $\mathcal{D}_{1,i}^w$ is at most $k-1$. Also for any $i>k$, the path from the root to $i^*$ corresponds to the path from the root to the corresponding node in the original Hoppe tree, i.e. the distance from the root to $i-k+1$, except that we might have an additional path among the first $k$ nodes instead of the first edge. Thus $\mathcal{D}^w_{1,i}$ is at least as big as $\mathcal{D}_{i-k+1}^\theta$.

Also $\mathcal{D}^w_{1,i}$ is at most  $k-1$ bigger than the distance between the root and the corresponding node in the original tree: Let $j-k+1$ be the first node on the path from 1 to $i-k+1$ in the original tree. Then in the reconstructed tree $j^*$ will be attached to some $h^*$, where $1\leq h \leq k$. Thus when we consider the path consisting of the path from the root to $h^*$ in the reconstructed tree, the edge from $h^*$ to $j^*$ and the path from $j-k+1$ to $i-k+1$ in the Hoppe tree, we get a path from $1^*$ to $i^*$ in the reconstructed tree. Thus  for all $k+1 \leq i \leq n$, there is some $h \leq k$ such that, 
\begin{equation*}
{D}^w_{1,i} = \mathcal{D}_{1,h}^w + 1 + \mathcal{D}_{j-k+1,i-k+1}^\theta = \mathcal{D}_{1,h}^w + \mathcal{D}_{1,i-k+1}^\theta.
\end{equation*}
Also $\mathcal{D}_{1,h}^w \leq k-1$, so we have
\begin{equation*}
\mathcal{D}_{1,i-k+1}^\theta \leq {D}^w_{1,i} \leq \mathcal{D}_{1,i-k+1}^\theta +k-1,
\end{equation*}
which implies that
\begin{equation*}
\begin{split}
\max_{i =1, \dots, n-k+1} \{\mathcal{D}_{1,i}^\theta\}  \leq \max_{i =1, \dots, n} \{\mathcal{D}_{1,i}^w\} \leq \max_{i =1, \dots, n-k+1} \{\mathcal{D}_{1,i}^\theta\} + k-1.
\end{split}
\end{equation*}
Thus from the definition of the height as $\mathcal{H}_n = \max_{i =1, \dots, n} \{\mathcal{D}_{1,i}\},$ we can derive that
\begin{equation} \label{eq:height}
\begin{split}
\mathcal{H}_{n-k+1}^\theta \leq \mathcal{H}_n^{w} \leq \mathcal{H}_{n-k+1}^\theta + k-1.
\end{split}
\end{equation}

We have the following result about the height of Hoppe trees.
\begin{theorem}[\cite{Hoppe}] \label{thm:HoppeHeight}
Let $\mathcal{H}_n^{\theta}$ denote the height of a Hoppe tree with $n$ nodes. Then
\begin{equation*}
\begin{split}
&\E [\mathcal{H}_n^{\theta}] = e\ln(n) - \frac{3}{2}\ln \ln n + \mathcal{O}(1) \text{ and} \\
&\Var(\mathcal{H}_n^{\theta}) = \mathcal{O}(1).
\end{split}
\end{equation*}
\end{theorem}

Together with the coupling this allows us to derive the following theorem.

\begin{theorem}
Let $\mathcal{H}^{w}_{n}$ denote the height of a WRT of size $n$ with weight sequence $(w_i)_{i \in \mathbb{R}}$ such that there is a $k \in \mathbb{N}$ such that for all $i>k$ we have $w_i = 1$. Then
\begin{equation*}
\begin{split}
&\E\left [\mathcal{H}_n^{w} \right ] = e \ln(n) - \frac{3}{2} \ln \ln (n) + \mathcal{O}(1) \text{ and }\\
&\Var(\mathcal{H}_n^{w}) = \mathcal{O}(1).
\end{split}
\end{equation*}
\end{theorem}
\begin{proof}
According to Theorem \ref{thm:HoppeHeight} we have 
$\E[\mathcal{H}_{n-k+1}^{\theta}] = e\ln(n-k+1) - \frac{3}{2}\ln \ln (n-k+1) + \mathcal{O}(1).$
Since,
\begin{equation*}
\begin{split}
& e\ln(n-k+1) - \frac{3}{2}\ln \ln (n-k+1) \\
& = e\left (  \ln(n) + \ln\left (1-\frac{k-1}{n} \right) \right ) \\
& \hspace{3em} - \frac{3}{2} \left ( \ln \ln(n) + \ln \left ( 1+ \frac{\ln\left (1-\frac{k-1}{n} \right )}{\ln(n)} \right ) \right ) \\
&= e \ln(n) - \frac{3}{2} \ln \ln (n) + o(1).
\end{split}
\end{equation*}
Thus we get by \eqref{eq:height} for $k>n$,
\begin{equation*}
e \ln(n) - \frac{3}{2} \ln \ln (n) + \mathcal{O}(1) \leq \E\left [\mathcal{H}_n^{w} \right ] \leq e \ln(n) - \frac{3}{2} \ln \ln (n) + \mathcal{O}(1)+ k-1
\end{equation*}
which  implies 
\begin{equation*}
\E\left [\mathcal{H}_n^{w} \right ] = e \ln(n) - \frac{3}{2} \ln \ln (n) + \mathcal{O}(1).
\end{equation*}
Similarly to before we might define $Y:= \mathcal{H}_n^{w} - \mathcal{H}_{n-k+1}^\theta$. Then by \eqref{eq:height} it holds that $Y<k$, so
\[\Var(\mathcal{H}_n^{w}) = \Var(\mathcal{H}_{n-k+1}^\theta+Y) = \Var(\mathcal{H}_{n-k+1}^\theta) + \Var(Y) + \Cov(\mathcal{H}_{n-k+1}^\theta, Y) = \mathcal{O}(1).\]

\hfill $\square$
\end{proof}

\section{Largest branch}\label{sec:largestbranch}

\subsection{Permutations view}

In this section, we focus on standard Hoppe trees and study the size of the largest branch in this   tree model. The results will sharpen the corresponding observations of \cite{Feng05} for URTs. Before moving on to largest branches, we need to discuss constructions of Hoppe trees via random permutations, in particular permutations that are generated via Ewens sampling formula.

For each $n$, Ewens distribution gives a family of distributions over the vectors $$C^{(n)} = (C_1^{(n)}, C_2^{(n)},\ldots,C_n^{(n)})$$ with $\sum_{i=1}^n i C_i^{(n)} = n.$  %The resulting distribution turns out to be the distribution of component counts for a random permutation of $n$ objects, chosen with probability biased by $\theta^{K}$, $\theta > 0$ is a parameter and where $K$  is the number of cycles in resulting permutation \cite{ABT}.  More specifically, t
In particular, for given $\theta > 0$, the Ewens distribution $EW(\theta)$ is defined with the probabilities $$\mathbb{P}_{\theta} (C^{(n)} = (c_1,\ldots,c_n)) = \mathbf{1}\left(\sum_{j=1}^n j c_j = n \right) \frac{n!}{\theta_{(n)}} \prod_{j=1}^n  \left( \frac{\theta}{j} \right)^{c_j} \frac{1}{c_j!}, \qquad  c_1,\ldots,c_n \in \mathbb{N},$$ where $\theta_{(n)} =  \theta(\theta + 1) \ldots (\theta + n - 1)$. 

Below we call a permutation resulting from $EW(\theta)$ a Hoppe permutation. In this setting, $c_i$ stands for the number of cycles in permutation of size $i$. In order to relate the topic to Hoppe trees, we need to discuss a recursive construction of Hoppe permutations. The discussion here  is similar to the one in   \cite{Altokislak}. For convenience,  the  permutation will be  constructed on the label set $\{2,3,\ldots,n\}$. We first begin with the  permutation $(2)$ with only one cycle. Then 3 either joins the first cycle to the right of $2$ with probability $\frac{1}{\theta + 1}$, or starts the second cycle with probability $\frac{\theta}{\theta + 1}$. Once we have constructed a permutation on $\{2,3,\ldots,k-1\}$, $k$ either starts a new cycle with probability $\frac{\theta}{\theta +k - 1}$, or is inserted to the right of a randomly chosen integer already assigned to a cycle. The resulting permutation then has the distribution $EW(\theta)$ \cite{ABT}. 

Next, we construct a Hoppe tree based on the permutation construction of the previous paragraph. Begin with node  1 as the root and node 2 attached to it. Then if node 3 begins a new cycle in the corresponding Hoppe permutation, attach it to node 1, and otherwise attach it to node 2. Then for node  $k \geq 4$, if $k$ starts a new cycle in the Hoppe permutation, then attach node $k$ to node $1$, and otherwise attach it to node $j$ where 
$k$ was inserted to the right of $j$ in the corresponding permutation.

It is then clear that this gives a bijection between Hoppe permutations and Hoppe trees. In particular, the cycles in a Hoppe permutation are in a one-to-one relation with the number of branches in the corresponding tree. This reduces the study of the size of the largest branch  in a Hoppe tree to the study of the largest cycle in its permutation correspondence.

\subsection{Size of the largest branch}

For a given tree $\mathcal{T}$ on $n$ vertices, the number of branches is the number of  children of the root, i.e. all nodes that are attached to the root. If a node $i$ is attached to the root, then node $i$ with its descendants is said to form a branch of the tree.  Let now  $\mathcal{B}_{n,i}(\mathcal{T})$  be the number of branches of size $i$ in $\mathcal{T}$. Also define   $$\nu_n(\mathcal{T}) := \max\{i \in [n-1]: \mathcal{B}_{n, i} \geq 1\}$$ to be the number of nodes  in the largest branch of a given tree, $\mathcal{T}$. In \cite{Feng05}, it was shown that  \begin{equation}\label{eqn:largestbranch}
\lim_{n \rightarrow \infty} \mathbb{P} \left(\nu_n(\mathcal{T}_n) \geq \frac{n}{2} \right) = \ln 2,
\end{equation} when $\mathcal{T}_n$ is a URT on $n$ vertices.    The first purpose of this section and the next theorem  is to  extend the  result of \cite{Feng05} to Hoppe trees, and to  provide more details about the asymptotic distribution, via exploiting the relation between Hoppe trees and Hoppe permutations.  Further, the result in \eqref{eqn:largestbranch} is now extended to an explicit expression for  $\lim_{n \rightarrow \infty} \mathbb{P} \left(\nu_n(\mathcal{T}_n) \geq cn \right)$ for $c \in [1/2,1]$. Once we have the results for the Hoppe tree case, the coupling construction of the previous section will also generalize these results to WRTs. 

\begin{theorem} 
(i) Let $\mathcal{T}_n^{\theta}$ be a Hoppe tree. Then $\frac{\nu_n (\mathcal{T}_n^{\theta})}{n}$ converges weakly to a random variable $\nu$ whose cumulative distribution function is given by
$$
F_{\theta}(x)  = e^{\gamma \theta} x^{\theta - 1} \Gamma(\theta) p_{\theta}(1 / x), \quad x > 0,$$
where   $\gamma$ is Euler's constant, 
$$p_{\theta}(x) = \frac{e^{- \gamma \theta} x^{\theta - 1}}{\Gamma(\theta)} \left(1 + \sum_{k=1}^{\infty} \frac{(- \theta)^k}{k!} \int \cdots \int_{\mathcal{S}_k(x)} \left(1 - \sum_{j=1}^k y_j \right)^{\theta - 1} \right) \frac{d y_1 \cdots d y_k}{y_1 \cdots y_k}, 
$$
with 
 $$\mathcal{S}_k(x)= \left\{y_1 > \frac{1}{x},\ldots, y_k > \frac{1}{x}, \sum_{j=1}^k y_j < 1\right\}.$$

(ii) When $\theta = 1$, we obtain the following for the largest branch in a URT:
 $\frac{\nu_n(\mathcal{T}_n)}{n}$ converges weakly to a random variable $\nu$ whose cumulative distribution function is given by
$$
F_1(x)  =
\begin{cases}
0, & \text{if } x<0 \\
1 + \sum_{k=1}^{\infty} \frac{(-1)^k}{k!} \int \cdots \int_{\mathcal{S}_k(x)} \frac{dy_1\ldots dy_k}{y_1\ldots y_k}, & \text{if } x \in [0,1] \\ 
1, & \text{if } x >1,
\end{cases}
$$
where 
  $\mathcal{S}_k(x)$ is as before. 

Also, for any $c \ in \left[\frac{1}{2},  1\right]$, we have $$\lim_{n \rightarrow \infty} \mathbb{P}(\nu_n(\mathcal{T}_n) \leq c n) = 1 - \ln (c^{-1}).$$ 
In particular,  $$\mathbb{E}[\nu] \geq \frac{n}{2}.$$
\end{theorem}

\begin{proof}
(i) First, we translate the problem into random permutation setting. We have 
$$\nu_n(\mathcal{T}_n^{\theta}) =_d \max\{i \in [n-1] : C_{n-1,i}(\theta) \geq 1\},$$ where $C_{n-1,i}(\theta)$ is the number of cycles of length $i$ in a $\theta$-biased Hoppe permutation. In this setting, Kingman \cite{Kingman77} shows that $\frac{\alpha_n(\theta)}{n}$  converges in distribution to a random variable $\alpha$ with cumulative distribution function $$F_{\theta} (x) = e^{\gamma \theta} x^{\theta - 1} \Gamma (\theta) p_{\theta} \left(\frac{1}{x} \right), x >0,$$ where $\gamma$ is Euler's constant, $$p_{\theta}(x) = \frac{e^{- \gamma \theta}}{\Gamma (\theta)} \left(1 + \sum_{k=1}^{\infty} \frac{(-\theta)^k}{k!} \right) \int \cdots \int_{\mathcal{S}_k(x)} \left(1 - \sum_{j=1}^k y_i \right)^{\theta - 1}\frac{dy_1\ldots dy_k}{y_1\ldots y_k}, $$ and  $$\mathcal{S}_k(x)= \left\{y_1 > \frac{1}{x},\ldots, y_k > \frac{1}{x}, \sum_{j=1}^k y_j < 1\right\}.$$
This proves the first part. 

(ii) Setting $\theta =1$ in argument of (i), and recalling that the random permutation in this case reduces to a uniformly random permutation   immediately  reveals the result. 

For the second claim, we first note Watterson \cite{Watterson76} shows that the derivative of $F_1(x)$ over $[1/2, 1]$ simplifies to $$f_1(x) = \frac{1}{x}.$$ Hence, for any $c \in \left[\frac{1}{2},  1\right]$, $$ \lim_{n \rightarrow \infty} \mathbb{P}(\nu_n(\mathcal{T}_n) \geq c n) = \mathbb{P}(\nu \geq c n) = \int_{c}^1 \frac{1}{x} dx = \ln (1/c).$$
Finally, we have  $$\mathbb{E}[\nu_n(\mathcal{T}_n)] \geq \int_{1/2}^1 x \frac{1}{x} dx = \frac{1}{2}.$$

\hfill $\square$
\end{proof}

\begin{remark}
The value $\lim_{n \rightarrow \infty} \frac{\nu_n(\mathcal{T}_n)}{n}$ when $\mathcal{T}_n$ is a URT is known to be the 
  Golomb-Dickman constant in the literature. Its exact value is known to be $0.62432998854...$.
\end{remark}

Now, let $\mathcal{T}_n^{w}$ be a  WRT of size $n$ with weight sequence $(w_i)_{i \in \mathbb{R}}$ such that there is a $k \in \mathbb{N}$ such that for all $i>k$ we have $w_i = 1$. Since we can  find a coupling of a Hoppe tree to a WRT in which the number of nodes in the largest branch  differs at most by $k$, the following now follows immediately. 
   
\begin{theorem}
Let $\mathcal{T}^{w}_{n}$ be a WRT of size $n$ with weight sequence $(w_i)_{i \in \mathbb{R}}$ such that there is a $k \in \mathbb{N}$ such that for all $i>k$ we have $w_i = 1$.  Then for any $c \ in \left[\frac{1}{2},  1\right]$, we have $$\lim_{n \rightarrow \infty} \mathbb{P}(\nu_n(\mathcal{T}_n^w) \leq c n) = 1 - \ln (c^{-1}).$$ 
\end{theorem}   
   
\section{Depth of node $n$}\label{sec:depth}

The depth of node $n$ is the length of the path from the root to $n$ or equivalently the number of ancestors of $n$. Note that in this and the next section we don't have any restrictions on the weight sequence $(w_i)_{i\in \mathbb{N}}$.

\begin{theorem}

Let $\mathcal{D}_n^{w}$ denote the depth of node $n$ in a WRT $\mathcal{T}_n^{w}$ and let $Z_n^w$ denote the set of ancestors of $n$. Let moreover $A_{i,n}^w := \1(i \in Z_n^w)$.	
Then 
\[\mathcal{D}_n^w = 1 + \sum_{i=2}^{n-1} A_{i,n}^w.
\]
The $A_{i,n}^w$ are mutually independent Bernoulli random variables with 
\[ \Pro( A_{i,n}^w ) = \frac{w_i}{\sum_{j=1}^{i}w_j}.\]

This directly yields the expectation and the variance:
\begin{equation*}
\E[\mathcal{D}_n^w] = \sum_{i=1}^{n-1} \frac{w_i}{\sum_{j=1}^{i}w_j}
\end{equation*}
and
\begin{equation*} \Var(\mathcal{D}_n^w) = \sum_{i=2}^{n-1} \frac{w_i}{\sum_{j=1}^{i}w_j} \left (1 - \frac{w_i}{\sum_{j=1}^{i}w_j} \right ).
\end{equation*}
\end{theorem}

\begin{proof}
Each claim will follow easily once we show that $\mathcal{D}_n^w$ can be written as a sum of independent Bernoulli random variables. For this purpose, we first observe that in a given rooted tree, the depth of a node is equal to its number of ancestors, since these determine the path from the root to the node. Using that $1$ definitely is an ancestor of $n$, in the notation of the theorem we thus get 
\[\mathcal{D}_n^w = 1 + \sum_{i=2}^{n-1} A_{i,n}^w.
\]
We will first find the distribution law of the $A_{i,n}^w$ and then show mutual independence. For the distribution law we will use the method used in \cite{Feng05}: we first find the values for $n-1$ and $n-2$ and then proceed by induction.

Node $n-1$ can only be an ancestor of $n$ if it is the parent of $n$, so we get 
\[
\Pro(n-1 \in Z_n^w) = \frac{w_{n-1}}{\sum_{i=1}^{n-1} w_i}.
\]
Similarly $n-2$ can only be an ancestor of $n$ if it is the parent of $n$ or it is the grandparent of $n$, in which case $n-2$ needs to be the parent of $n-1$ who needs to be the parent of $n$. This gives
\begin{equation*}
\begin{split}
\Pro(n-1 \in Z_n^w) = \underbrace{\frac{w_{n-2}}{\sum_{i=1}^{n-1} w_i}}_{n-2 \text{ is parent of } n} + \underbrace{\frac{w_{n-2}}{\sum_{i=1}^{n-2} w_i}\frac{w_{n-1}}{\sum_{i=1}^{n-1} w_i}}_{n-2 \text{ is grandparent } of n} 
= \frac{w_{n-2}}{\sum_{i=1}^{n-2} w_i}.
\end{split}
\end{equation*}
We will now show by induction that  for all $j = 2, \dots, n-1$,
\[\Pro( j \in Z_n^w) = \frac{w_j}{\sum_{i=1}^{j} w_j}.
\]
Let the above be true for all $j \geq i+1$ and let $C_{i,j}^w$ denote the event that $j$ is a child of $i$. Then 
\[\Pro( i \in Z_n^w) = \sum_{j = i+1}^{n-1} \Pro( j \in Z_n^w, C_{i,j}^w) + \Pro(C_{i,n}^w).
\]

Since $C_{i,j}^w$ only relates to the $j^{th}$ step of the construction process and $j \in Z_n^w$ only depends on the $j+1^{th}, \dots, n^{th}$ step, these two events are independent. We thus get
\[\Pro(i \in Z_n^w) =  \sum_{j = i+1}^{n-1} \Pro( j \in Z_n^w)\Pro(C_{i,j}^w)+\Pro(C_{i,n}^w) = \sum_{j = i+1}^{n-1} \left (\frac{w_j}{\sum_{k=1}^{j}w_k} \frac{w_i}{\sum_{k=1}^{j-1}w_k} \right )+ \frac{w_i}{\sum_{j=1}^{n-1}w{j}}.
\]

To simplify this expression we first note that we can factor out $w_i$ and by some elementary operations get:
\begin{equation*}
\begin{split} 
 \frac{w_{i+1}}{\sum_{k=1}^{i+1}w_k\sum_{k=1}^{i}w_k} + \frac{w_{i+2}}{\sum_{k=1}^{i+2}w_k\sum_{k=1}^{i+1}w_k} = \frac{w_{i+1}+ w_{i+2}}{\sum_{k=1}^{i+2}w_k\sum_{k=1}^{i}w_k}.
\end{split}
\end{equation*}

In general, the following holds for $l \in \mathbb{N}$:
\begin{equation*}
\begin{split}
&\frac{w_{i+1}+ w_{i+2} + \cdots + w_{i+l}}{\sum_{k=1}^{i} w_k \sum_{k=1}^{i+l} w_k} + \frac{w_{i+l+1}}{\sum_{k=1}^{i+l+1} w_k \sum_{k=1}^{i+l} w_k} = \frac{w_{i+1}+  \cdots + w_{i+l+1}} {\sum_{k=1}^{i+l+1} w_k  \sum_{k=1}^{i} w_k}.
\end{split}
\end{equation*}

By using this equality $n-i-2$ times, we thus get 
\begin{equation*}
\begin{split}
\Pro(i \in Z_n^w) &= w_i \left ( \frac{w_{i+1} + \cdots + w_{n-1}}{\sum_{k=1}^{i}w_k \sum_{k=1}^{n-1}w_k} + \frac{1}{\sum_{k=1}^{n-1}w_k} \right ) = \frac{w_i}{\sum_{k=1}^{i} w_k}.
\end{split}
\end{equation*}
Now we will show that the events $A_{i,n}^w$ are mutually independent for $j = 2, \dots, n-1$. For this we will use the method used in \cite{Hoppe}: for any $2 \leq k \leq n-2$ and $2 \leq j_k < \dots < j_2 < j_1 \leq n-1$ consider the event that all $j_i$'s and only the $j_i$'s are ancestors of $n$. We will denote this event by $E$. Then
\[E := ( \1(j_i \in Z_n^w) =1, \1(j \in Z_n^w) = 0, \text{ for } j \neq j_i, i=1, \dots, k).
\]
By the structure of the recursive tree, to realize this event, $n$ must be a child of $j_1$, $j_1$ a child of $j_2, \dots , j_{k-1}$ a child of $j_k$ and $j_k$ a child of 1. In general for $i=1, \dots, k-1$, $j_i$ must be a child of $j_{i+1}.$  It does not matter what nodes  $j \neq j_i$ attach to. Hence, by the attachment probabilities we get:
\begin{equation*}
\begin{split}
\Pro(E) & = \Pro(j_i \in Z_n^w, j \notin Z_n^w, \text{ for } j \neq j_i, i=1, \dots, k) \\
&= \underbrace{\frac{w_{j_1}}{\sum_{\ell=1}^{n-1} w_{\ell}}}_{n \text{ child of } j_1} \prod_{i=1}^{k-1} \underbrace{\frac{w_{j_{i+1}}}{\sum_{\ell=1}^{j_i-1}w_{\ell}}}_{j_i \text{ child of } j_{i+1}} \underbrace{\frac{w_1}{\sum_{\ell=1}^{j_k-1}w_{\ell}}}_{j_k \text{ child of } 1} \\
&= w_1 w_{j_1} \cdots w_{j_k} \prod_{i=1}^{n-1} \frac{1}{\sum_{\ell=1}^{i}w_{\ell}}  \prod_{\substack{1<j<n \\ j \neq j_i, i=1, \dots, k}} \left (\sum_{\ell=1}^{j-1}w_{\ell} \right ) \\
& = \prod_{i=1}^{k}\frac{w_{j_i}}{\sum_{\ell=1}^{j_i}w_{\ell}} \prod_{\substack{1<j<n \\ j \neq j_i, i=1, \dots, k}} \left ( \frac{\sum_{\ell=1}^{j-1}w_{\ell}}{\sum_{\ell=1}^{j}w_{\ell}}\right )  \frac{w_1}{\sum_{\ell=1}^{1}{w_{\ell}}} \\
&= \prod_{i=1}^{k}\frac{w_{j_i}}{\sum_{\ell=1}^{j_i}w_{\ell}} \prod_{\substack{1<j<n \\ j \neq j_i, i=1, \dots, k}} \left ( 1 - \frac{w_j}{\sum_{\ell=1}^{j}w_{\ell}}\right )  \\
&= \prod_{i=1}^{k} \Pro(j_i \in Z_n^w) \prod_{\substack{1<j<n \\ j \neq j_i, i=1, \dots, k}} \Pro(j \notin Z_n^w).
\end{split}
\end{equation*}

This implies that the events $\1(i \in Z_n^w)$ are mutually independent. Hence the $A_{i,n}^w$ are mutually independent Bernoulli random variables. 

Expectation and variance formulas for  $\mathcal{D}_{n}^w$ follows from this observation right away. 

\hfill $\square$
\end{proof}

The following central limit theorem now follows. 

\begin{theorem}
If $\E[\mathcal{D}_n^w ] \rightarrow \infty$ and $\limsup_{n \to \infty} \frac{w_n}{\sum_{i=1}^{n} w_i} < 1$, then we have $$\frac{\mathcal{D}_n^w - \mathbb{E}[\mathcal{D}_n^w ]}{\sqrt{Var(\mathcal{D}_n^w )}} \longrightarrow_d \mathcal{G} \quad \text{as} \quad n \rightarrow \infty.$$
\end{theorem}

\begin{proof}
By Liapounov's central limit theorem for sums of independent Bernoulli random variables, if $\Var({D}_n^w) \to \infty$, it holds that $\frac{\mathcal{D}_n^w - \mathbb{E}[\mathcal{D}_n^w ]}{\sqrt{Var(\mathcal{D}_n^w )}} \longrightarrow_d \mathcal{G} \quad \text{as} \quad n \rightarrow \infty$. Now let $p_n = \frac{w_n}{\sum_{j=1}^{n} w_j}$.
Since  $\limsup_{n \to \infty} p_n <1$, there is an $0<\varepsilon<1$ and an $N \in \mathbb{N}$ such that for all $n >N$, $ 1- p_n > \varepsilon$. 
Then \[\Var({D}_n^w) = \sum_{i=1}^{n} p_i(1-p_i) = \sum_{i=1}^{N} p_i(1-p_i) + \sum_{i=N+1}^{n} p_i(1-p_i) >  \varepsilon \sum_{i=N+1}^{n} p_i. \] Since we know that $\E[\mathcal{D}_n^w ] = \sum_{i=1}^{n} p_i \rightarrow \infty$ as $n \to \infty$ this implies that $\Var(\mathcal{D}_n^w) \to \infty$ as $n \to \infty$. 

\hfill $\square$
\end{proof}

\begin{example}
(i)  If the weights are limited from below and above, the expectation and variance of the depth of node $n$ will still be equal to $\mathcal{O}(\ln(n))$ asymptotically.

(ii) For the Hoppe tree, we write $\mathcal{D}_n^\theta$ for the depth of node $n$ and have 
\begin{equation*}
\begin{split}
&\E[\mathcal{D}_n^\theta] = 1+ \sum_{i=1}^{n-2} \frac{1}{\theta+1} = \log(n) + \mathcal{O}(1) \text{ and }\\
&\Var \left (\mathcal{D}_n^\theta \right ) = \sum_{i=1}^{n-2} \frac{\theta+i-1}{(\theta+i)^2} = \log(n) + \mathcal{O}(1).
\end{split}
\end{equation*}
Notice that the depth in a URT and a Hoppe tree asymptotically equivalent. The same conclusion also holds for a wide range of statistics whose dependence to the root is  small and vanishes asymptotically. This makes the asymptotic study of Hoppe trees slightly uninteresting. Another such example  is the number of leaves in a Hoppe tree which was studied earlier.

(iii) Let's mention one instance where the behavior of the depth is totally different than the URT case. Let $(w_i)_{i\in \mathbb{N}} = \left (\frac{1}{i^2} \right )_{i \in \mathbb{N}}$. Then $
\E[\mathcal{D}_n^w] =  \sum_{i=1}^{n-1} \frac{\frac{1}{i^2}}{\sum_{j=1}^{i}\frac{1}{j^2}}$,
and since for all $i \in \mathbb{N}$ we have $1<\sum_{j=1}^{i}\frac{1}{j^2}<\frac{\pi^2}{6}$, we get
\begin{equation*}
\begin{split}
  \frac{6}{\pi^2} \sum_{i=1}^{n-1} \frac{1}{i^2} \leq \E[\mathcal{D}_n^w] \leq \sum_{i=1}^{n-1} \frac{1}{i^2}
  \Longrightarrow \frac{6}{\pi^2}  \leq \E[\mathcal{D}_n^w] \leq \frac{\pi^2}{6}.
\end{split}
\end{equation*} Also, since  
$
\Var(\mathcal{D}_n^w) = \sum_{i=2}^{n-1} \frac{\frac{1}{i^2} \sum_{j=1}^{i-1} \frac{1}{j^2}}{\left (\sum_{j=1}^{i}\frac{1}{j^2}\right )^2},
$
and  for all $i\geq 2$, it holds that $\frac{4^2}{5^2} \leq \frac{\sum_{j=1}^{i-1} \frac{1}{j^2}}{\left (\sum_{j=1}^{i}\frac{1}{j^2}\right )^2}\leq 1$, we get
\begin{equation*}
\begin{split}
\frac{4^2}{5^2} \sum_{i=2}^{n-1} \frac{1}{i^2} \leq \Var(\mathcal{D}_n^w) \leq \sum_{i=2}^{n-1} \frac{1}{i^2} 
\Longrightarrow \frac{4}{5^2}  \leq \Var(\mathcal{D}_n^w) \leq \frac{\pi^2}{6} -1.
\end{split}
\end{equation*}
\end{example}

\section{Number of branches}\label{sec:branches}

Finally, we study the number of branches in a WRT. Our results are  summarized in the following theorem. 

\begin{theorem}
Let $\{w_n\}_{n \geq 1}$ be a sequence of positive weights.  Denote the number of branches in $\mathcal{T}_n^w$ by $\mathcal{B}_n^w$.  
\begin{itemize}
\item[(i)]  We have \begin{equation*} \mathbb{E}[\mathcal{B}_n^w] = \sum_{i=1}^{n-1} \frac{w_1}{\sum_{k=1}^i w_k}  \qquad \text{and} \qquad Var(\mathcal{B}_n^w) = \sum_{i=2}^{n-1} \frac{w_1}{\sum_{k=1}^i w_k}  \left ( 1-\frac{w_1}{ \sum_{k=1}^i w_k }\right ).
\end{equation*}
\item[(ii)] If $\E [\mathcal{B}_n^w]$ diverges, then the central limit $$\frac{\mathcal{B}_n^w - \mathbb{E}[\mathcal{B}_n^w]}{\sqrt{Var(\mathcal{B}_n^w)}} \longrightarrow_d \mathcal{G} \quad \text{as} \quad n \rightarrow \infty.$$ 

\item[(iii)] Further, one has $$d_W \left(\frac{\mathcal{B}_n^w - \mathbb{E}[\mathcal{B}_n^w]}{\sqrt{Var(\mathcal{B}_n^w)}}, \mathcal{G} \right) \leq \frac{1}{\sqrt{Var(\mathcal{B}_n^w)}} \frac{\sqrt{28} + \sqrt{\pi}}{\sqrt{\pi}}.$$
\end{itemize}
\end{theorem}

\begin{proof}
Letting $b_i = \1(\text{node $i$ attaches to node 1})$, observe that the $b_i$'s are independent and that $\E[b_i]= \frac{w_1}{\sum_{j=1}{i-1} w_j}$ . The formula for $\E[\mathcal{B}_n^w] = \sum_{i=2}^{n} \E[b_i]$ and $\Var(\mathcal{B}_n^w)= \sum_{i=2}^{n} \Var(b_i)$ is then clear.

The CLT then follows from Liapounov's central limit theorem. Let $\E[\mathcal{B}_n^w] \to \infty$. To show that this implies that $\Var(\mathcal{B}_n^w)  \to \infty$ too, we differentiate two cases:
\begin{itemize}
\item[(a)] If $\sum_{i=1}^{n} w_i \to c \in \mathbb{R}$ as $n \to \infty$, $ \frac{\sum_{i=2}^{n} w_i}{(\sum_{i=1}^{n} w_i)^2} \to \frac{c-w_1}{c^2}$ and hence $\Var(\mathcal{B}_n^w) $ diverges.
\item[(b)]  If $\sum_{i=1}^{n} w_i \to \infty$ as $n \to \infty$,this implies that $\frac{w_1}{\sum_{i=1}^{n} w_i} \to 0$ as $n \to \infty$. Hence there is an $N \in \mathbb{N}$ such that for all $n >N$, we have $\frac{w_1}{\sum_{i=1}^{n} w_i}< \frac{1}{2}.$ Then 
\[\Var(\mathcal{B}_n^w) > \sum_{i=N+1}^{n-1} \frac{w_1}{\sum_{k=1}^i w_k}  \left ( 1-\frac{w_1}{ \sum_{k=1}^i w_k }\right ) > \frac{1}{2} \sum_{i=N+1}^{n-1} \frac{w_1}{\sum_{k=1}^i w_k}  \to \infty.\]
\end{itemize}

The convergence rates can be obtained by using Theorem 3.1 in \cite{Ross}. 
\hfill $\square$
\end{proof}

\begin{remark}
(i) When $w_1 = \theta >  0$, and $w_i = 1$, $i = 1,2,\ldots$, we obtain $$\frac{\mathcal{B}_n - \theta \ln n}{\sqrt{\theta \ln n}} \longrightarrow_d \mathcal{G}, \quad \text{as} \quad n \rightarrow \infty.$$ In particular, when $\theta = 1$ as well, one recovers the central limit theorem for the URT case. 

(ii) When there exist $\alpha_1, \alpha_2$ so that $0 < \alpha_1 \leq \sup_i w_i \leq \alpha_2 < \infty$, it can be shown that both the expectation and the variance are still of order $n$. 

(iii) When $w_i$'s are not bounded, there can be big differences compared to the case of URTs. One such extreme case is when $w_i  =i$ where $\mathbb{E}[\mathcal{B}_n^w] \sim 2$, and $Var(\mathcal{B}_n^w) \sim 14 - \frac{4 \pi^2}{3}$. Another one is when $w_i = 1 / i$ in which case  $$\frac{6}{\pi^2} (n - 1) \leq \mathbb{E}[\mathcal{B}_n^w] \leq n -1, \qquad \text{and} \qquad \frac{6}{5\pi^2} (n - 2) \leq Var(\mathcal{B}_n^w) \leq \frac{1}{4} (n-2).$$

(iv) It is well known that the number of branches and the depth of node $n$ have the same distribution in a URT. This is not the case when the tree is non-uniform. Indeed, this is intuitively clear since having more branches increases the chance that node $n$ attaches to a node at a lower level.

\end{remark}

\end{document}